\documentclass{amsart}
\usepackage{amssymb}
\usepackage{amsthm}
\usepackage{amsmath}
\usepackage{amscd}
\usepackage{amsfonts}
\usepackage{amsrefs}

\newtheorem{theorem}{Theorem}[section]
\newtheorem{lemma}[theorem]{Lemma}

\theoremstyle{definition}

\newtheorem{example}{Example}

\newtheorem{remark}{Remark}

\begin{document}
	
\title{The Euler Characteristic of Finite Categories}
\author{Mustafa Akkaya}
\address{\vbox{\hbox{Department of Mathematics, Bilkent University} \hbox{Çankaya, Ankara, Turkey}}}
\email{mustafa.akkaya@bilkent.edu.tr}

\author{\"Ozg\"un \"Unl\"u}
\address{\vbox{\hbox{Department of Mathematics, Bilkent University} \hbox{Çankaya, Ankara, Turkey}}}
\email{unluo@fen.bilkent.edu.tr}

\maketitle

\begin{abstract}
We associate a rational number $\chi(\mathcal{A})$ to every category $\mathcal{A}$ whose object and morphism sets are finite. We show that the assignment $\chi $ is additive under disjoint union and it preserves products. Hence we consider $\chi $ as an Euler measure on a family of categories where this assignment also obeys a version of the inclusion–exclusion principle.
\end{abstract}

\section{Introduction}
We associate a rational number $\chi(\mathcal{A})$ to every category $\mathcal{A}$ whose object and morphism sets are finite. We show that the assignment $\chi $ is additive under disjoint union (Theorem \ref{thm:disjoint}) and it preserves products (Theorem \ref{thm:product}). Hence, as in \cite{propparxiv} and \cite{propp}, we call $\chi(\mathcal{C})$ the Euler measure of $\mathcal{C}$ for a family of categories where this assignment also obeys a version of the inclusion–exclusion principle (Theorem \ref{thm:inc_exc}).

Let  $\chi_{Lein}$ be the Euler characteristic defined by Leinster in \cite{leinster}. We show (Theorem \ref{thm:ext_of_Lein}) that $\chi $ extends $\chi_{Lein}$ to all categories.  More explicitly, to define $\chi_{Lein}$ Leinster requires the existence of both a weighting and a coweighting  on the adjacency matrix of the directed multi-graph associated to the category. Here we could define $\chi $ of a finite category without any condition. However, to prove that $\chi $ is invariant under categorical adjunctions (Theorem \ref{thm:equiv_of_cat}), we need to put the same conditions as Leinster. Moreover in Example \ref{EX2}, we show that these conditions cannot be eliminated. Hence we call $\chi $ an Euler characteristic only when it agrees with $\chi_{Lein}$. 

Leinster's formula (reproved here as Theorem \ref{thm:inc_exc_with_coweighting}) for computing the Euler characteristic of the Grothendieck construction of a diagram of categories requires the existence of a coweighting on the matrix associated with the Grothendieck construction. Here we show that this condition can be eliminated under some extra conditions on the diagram of categories. As a result we can consider $\chi $ as an Euler measure on a larger family of categories closed under Grothendieck constructions over posets. However, in Example \ref{EX3} we show that the poset condition cannot be eliminated.

In Section \ref{sec:appendix}, we discuss a method for computing $\chi $ to prove that it always assigns a rational number to every finite category.

Here, we also note that Stephanie Chen and Juan Pablo Vigneaux independently came upon the idea of using the Moore-Penrose inverse to study categorical magnitude in \cite{chen-vigneaux} at about the same time as us.


\section{Definition of Euler Characteristic}

All matrices in this paper have entries in $\mathbb{Q}$. For an $m\times n$-matrix $M$, let $M^+$ be the Moore–Penrose inverse of $M$ and $M^*$ be the transpose of $M$.

\subsection{Properties of Moore-Penrose inverse}
In this subsection we list some well known properties of Moore-Penrose inverse that we use throughout the paper.
Let $M$ be an $m\times n$-matrix. It is known (see \cite{moore}, \cite{penrose}, and \cite{israel}) that the \textbf{Moore–Penrose inverse} of the $m\times n$-matrix $M$ is the unique $n\times m$-matrix $M^+$ satisfying the following four equations:
\begin{equation}\label{eqn:1}MM^+M=M\end{equation}
and
\begin{equation}\label{eqn:2}M^{+}MM^{+}=M^{+}\end{equation}
and
\begin{equation}\label{eqn:3}(M^+M)^*=M^+M\end{equation}
and
\begin{equation}\label{eqn:4}(MM^+)^*=MM^+.\end{equation}

For an $m\times n$-matrix $M$, the \textbf{row space} $\mathrm{Row}(M)$ of $M$ is the subspace of $1\times n$-matrices spanned by rows of $M$. Considering $M$ as a linear function from the vector space of $n\times 1$-matrices to the vector space of $m\times 1$-matrices, the \textbf{null space} $\mathrm{Null}(M)$ of $M$ is the kernel of $M$  and the \textbf{image} $\mathrm{im}(M)$ of $M$  is its image. For a subspace $V$ of the vector space of $m\times 1$-matrices, \textbf{projection map} $\mathrm{Proj}_{V}$ \textbf{onto} $V$ is the orthogonal projection from the vector space of $m\times 1$-matrices onto $V$ where we consider the standard inner product on the vector space of $m\times 1$-matrices. For a subspace $V$ of the vector space of $m\times 1$-matrices. Let $V^{\perp}$ denote the orthogonal complement of $V$.
\begin{lemma}\label{lemma_about_moore_penrose_inverse}
	Let $x$ be an $m\times 1$-matrix, $y$  an $n\times 1$-matrix, and $M$ an $m\times m $-matrix. Then $M^+x=y$ if and only if $y^*\in Row(M)$ and $My=\mathrm{Proj}_{im(M)}(x)$.
\end{lemma}
\begin{proof}
Notice that $M^+x=y$ implies $y=M^+My$  by (\ref{eqn:2}). Multiplying $M^+x=y$ by $M$ from left gives us $My=MM^+x$. Conversely, if we assume $y=M^+My$ and  $My=MM^+x$ then we have $M^+x=M^+MM^+x=M^+My=y$. Hence we have showed 
\begin{eqnarray}\label{equiv}
	M^+x=y \text{\ \ if and only if\ \ } y=M^+My \text{\ \ and\ \ } My=MM^+x.
\end{eqnarray}
 
Let $P=M^+M$. Then $P^2=P$ by (\ref{eqn:1}) and $P^*=P$ by (\ref{eqn:3}). So $P$ is an orthogonal projection. Notice that $M(I-P)=0$ by (\ref{eqn:1}). So $\mathrm{im}(I-P)\subseteq \mathrm{Null}(M)$. We also have $P=M^*(M^+)^*$. Hence $\mathrm{im}(P)\subseteq \mathrm{im}(M^*)$. Since $\mathrm{im}(M^*)=\mathrm{Null}(M)^{\perp}$, we have $$M^+M=P=\mathrm{Proj}_{\mathrm{im}(M^*)}.$$
Also notice that $(M^*)^+=(M^+)^*$. Hence we have
$$MM^+=(MM^+)^*=(M^+)^*M^*=(M^*)^+M^*=\mathrm{Proj}_{\mathrm{im}(M^{**})}=\mathrm{Proj}_{\mathrm{im}(M)}.$$
 The statement $y^*\in Row(M)$ is equivalent to the statement $y\in \mathrm{im}(M^*)$. Also notice that we have  $(I-M^+M)y=0$ if and only if $y\in \mathrm{im}(M^*)$ since $M^+M=\mathrm{Proj}_{\mathrm{im}(M^*)}.$ Therefore we have  $y^*\in Row(M)$ if and only if $y=M^+My$. Hence the equivalence in (\ref{equiv}) becomes
$$M^+x=y\text{\ \  if and only if \ \  } y^*\in Row(M) \text{ and } My=\mathrm{Proj}_{im(M)}(x)
.$$
\end{proof}

Let $P$ be an invertible $m\times m$-matrix. Then $P^+=P^{-1}$. In particular, if  $P$ is an orthogonal matrix then  $P^+=P^*$. Moreover for orthogonal matrices we have the following lemma.

\begin{lemma}\label{lemma_about_orthogonal_matrices}

Let $P$ be an orthogonal $m\times m$ matrix  and $M$ be any $m\times m$ matrix. Then $(MP)^+=P^+M^+$ and $(PM)^+=M^+P^+$.
\end{lemma}

\begin{proof}
Since $P$ is orthogonal, we have $P^+ = P^*$. The matrix $P^+ M^+$ satisfies Equation (\ref{eqn:1}), because $$(MP)(P^+ M^+)(MP)=M M^+ M P= MP.$$ The matrix  $P^+ M^+$ satisfies Equation (\ref{eqn:2}), because $$(P^+ M^+)(MP)(P^+ M^+)=P^+ M^+ M M^+ = P^+ M^+.$$
The matrix $P^+ M^+$ satisfies Equation (\ref{eqn:3}), because $$((P^+ M^+)(MP))^*=P^* M^* (M^+)^* (P^+)^*= P^+ (M^+ M)^* P = (P^+M^+)(MP).$$
The matrix $P^+ M^+$ satisfies Equation (\ref{eqn:4}), because $$((MP)(P^+ M^+))^*= (M M^+)^* = M M^+ = (MP)(P^+ M^+).$$
Similarly, $(PM)^+=M^+P^+$.

\end{proof}
Let $M\otimes N$ denote the Kronecker product of the square matrices $M$ and $N$. Then we know that
$$(M\otimes N)^+=M^+\otimes N^+$$
because in general we have $(A \otimes B)(C\otimes D) = (AC) \otimes (BD)$ when $A$, $B$, $C$, $D$ are square matrices.

Let $A_1$, $A_2$, $\dots$ $A_m$ be square matrices of possibly different sizes.
Then the matrix $\mathrm{diag}(A_1,A_2,\dots,A_n)$ is defined as the square matrix obtained by putting  $A_1$, $A_2$, $\dots$ $A_m$ on the diagonal and zero everywhere else. Then we know that
$$\mathrm{diag}(A_1,A_2,\dots,A_n)^+=\mathrm{diag}(A_1^+,A_2^+,\dots,A_n^+).$$

\subsection{Definition of $\chi(\mathcal{A})$}
Let $\mathbf{1}_m$ denote the $m\times 1$-matrix whose all entries are equal to $1$. Given a finite category $\mathcal{A}$, let $[\mathcal{A}]$ denote the adjacency matrix associated with the directed multigraph obtained from the category $\mathcal{A}$ by forgetting its composition data. In other words, if the object set of $\mathcal{A}$ is $\{a_1,a_2,\dots , a_m\}$ then $[\mathcal{A}]$ is the  $m\times m$-matrix whose $ij^\text{th}$ entry is the number of elements of the morphism set $\mathcal{A}(a_i,a_j)$. Let $|\mathcal{A}|$ denote the number of objects of a category $\mathcal{A}$. We define a rational number $\chi(\mathcal{A})$ associated with a category $\mathcal{A}$ as follows:
$$\chi(\mathcal{A})=\mathbf{1}_{|\mathcal{A}|}^*[\mathcal{A}]^+\mathbf{1}_{|\mathcal{A}|}$$
where we consider the $1\times 1$ matrix on the right hand side as a rational number. Note that $\chi(\mathcal{A})$ is independent of the order chosen on the set of objects of $\mathcal{A}$. To see this let $P$ be any ${|\mathcal{A}|}\times {|\mathcal{A}|}$-permutation matrix. Then $P$ is an orthogonal matrix and hence by Lemma \ref{lemma_about_orthogonal_matrices} we have
$$\mathbf{1}_{|\mathcal{A}|}^*(P^{-1}[\mathcal{A}]P)^+\mathbf{1}_{|\mathcal{A}|}=\mathbf{1}_{|\mathcal{A}|}^*P^{-1}[\mathcal{A}]^+P\mathbf{1}_{|\mathcal{A}|}=\mathbf{1}_{|\mathcal{A}|}^*[\mathcal{A}]^+\mathbf{1}_{|\mathcal{A}|}.$$
This proves that $\chi(\mathcal{A})$ is independent of the order of objects of $\mathcal{A}$.

\subsection{Comparison with $\chi_{Lein}(\mathcal{A})$}

We say an  $n\times 1$-matrix $w$ is \textbf{weighting} for an $m\times n$-matrix $M$ if $$Mw=\mathbf{1}_m.$$ An  $1\times m$-matrix $v$ is called a \textbf{coweighting }for an $m\times n$-matrix $M$ if $$vM=\mathbf{1}_n^*.$$
We will use the following facts about weightings and coweightings.
\begin{lemma}\label{lemma_for_existence_of_weighting}
	Let $M$ be an $m\times n$-matrix. Then  $M^+\mathbf{1}_m$ is a weighting of $M$  if and only if  $M$ has a weighting.
\end{lemma}
\begin{proof}
	Assume that $w$ is a weighting of $M$. Then $$\mathbf{1}_m=Mw=MM^+Mw=MM^+\mathbf{1}_m.$$	
	Hence $M^+\mathbf{1}_m$	is a weighiting for $M$. The converse is clear.
\end{proof}
\begin{lemma}\label{lemma_for_existence_of_coweighting}
	Let $M$ be an $m\times n$-matrix. Then  $\mathbf{1}_n^*M^+$ is a coweighting of $M$  if and only if  $M$ has a coweighting.
\end{lemma}
\begin{proof}
	Assume that $v$ is a coweighting of $M$. Then $$\mathbf{1}_n^*=vM=vMM^+M=\mathbf{1}_n^*M^+M.$$	
	Hence $\mathbf{1}_n^*M^+$	is a coweighiting for $M$. The converse is clear.	
\end{proof}

\begin{lemma}\label{lemma_for_SLS}
	Let $M$ be an $m\times n$-matrix.  Assume that $w$ is a weighting for $M$ and $v$ is a coweighting for $M$. Then
	$$ \mathbf{1}_n^*w=v\mathbf{1}_m=\mathbf{1}_n^*M^+\mathbf{1}_m.$$
\end{lemma}
\begin{proof} Assume that $w$ is a weighting for $M$ and $v$ is a coweighting for $M$. Then we have
	$$\mathbf{1}_n^*w=vMw=v\mathbf{1}_m.$$
To prove the last equality, note that	
	$$vMw=vMM^+Mw=\mathbf{1}_n^*M^+\mathbf{1}_m.$$
\end{proof}

Let $\chi_{Lein}(\mathcal{A})$  denote the Euler characteristic of the category $\mathcal{A}$ as defined in \cite{leinster}.
We now show that $\chi$ is an extension of $\chi_{Lein}$ to all finite categories.

\begin{theorem}\label{thm:ext_of_Lein}
	Let $\mathcal{A}$ be a finite category and assume that $\chi_{Lein}(\mathcal{A})$ is defined. Then $$\chi_{Lein}(\mathcal{A})=\chi (\mathcal{A}).$$
\end{theorem}

\begin{proof} Let $\mathcal{A}$ be a finite small category. Assume that the Euler characteristic $\chi_{Lein}(\mathcal{A})$ of the category $\mathcal{A}$ is defined as in \cite{leinster}. This means that there exists a coweighting $v$ and a weighting $w$ for $[\mathcal{A}]$.  Hence by Lemma \ref{lemma_for_SLS} we have	 $$\chi_{Lein}(\mathcal{A})=\mathbf{1}_{|\mathcal{A}|}^*w=\mathbf{1}_{|\mathcal{A}|}^*[\mathcal{A}]^+\mathbf{1}_{|\mathcal{A}|}=\chi(\mathcal{A}).$$
\end{proof}

\subsection{Properties of $\chi $}
For a functor $L:\mathcal{A}\rightarrow \mathcal{B}$, let $[L]$ denote the $|\mathcal{B}|\times |\mathcal{A}|$-matrix whose $ij^{\text{th}}$ entry is $1$ if $L$ sends the $j^{\text{th}}$ object of $\mathcal{A}$ to the $i^{\text{th}}$ object of $\mathcal{B}$ and $0$ otherwise. We give a proof of Proposition 2.4 in \cite{leinster} by using the terminolgy developed above.
\begin{theorem}\label{thm:equiv_of_cat}
	Let $L:\mathcal{A}\rightleftarrows B:R $ be an adjunction. Assume that $[\mathcal{A}]$ has a coweighting and $[\mathcal{B}]$ has a weighting.
	Then $[\mathcal{B}]$ has a coweighting and $[\mathcal{A}]$ has a weighting  and we have $\chi (\mathcal{A})= \chi (\mathcal{B})$.
\end{theorem}
\begin{proof}
	Assume that $[\mathcal{A}]$ has a coweighting and $[\mathcal{B}]$ has a weighting and $L:\mathcal{A}\rightleftarrows B:R $ is an adjunction. By Lemma \ref{lemma_for_existence_of_weighting} and Lemma \ref{lemma_for_existence_of_coweighting}, we know that
	 $$\mathbf{1}_{|\mathcal{A}|}^*[\mathcal{A}]^+[\mathcal{A}]=\mathbf{1}_{|\mathcal{A}|}^*\text{\ \ and \ \ } [\mathcal{B}][\mathcal{B}]^+\mathbf{1}_{|\mathcal{B}|}=\mathbf{1}_{|\mathcal{B}|}.$$
	 Moreover, we know that
	 $$\mathbf{1}_{|\mathcal{B}|}^*=\mathbf{1}_{|\mathcal{A}|}^*[R]
	 \text{\ \ and \ \ }[L]^*\mathbf{1}_{|\mathcal{B}|}=\mathbf{1}_{|\mathcal{A}|} $$
	 because each column of the matrices $[R]$ and $[L]$ contains exactly one entry that is equal to $1$. Finally we have
	$$[\mathcal{A}][R]=[L]^* [\mathcal{B}]$$
	because $R$ and $L$ are adjoints of each other. Now notice that $$\mathbf{1}_{|\mathcal{B}|}^*=\mathbf{1}_{|\mathcal{A}|}^*[R]=\mathbf{1}_{|\mathcal{A}|}^*[\mathcal{A}]^+[\mathcal{A}][R]=\mathbf{1}_{|\mathcal{A}|}^*[\mathcal{A}]^+[L]^*[\mathcal{B}].$$
	Hence $\mathbf{1}_{|\mathcal{A}|}^*[\mathcal{A}]^+[L]^*$ is a coweighting for  $[\mathcal{B}]$. Similarly, $[R][\mathcal{B}]^+\mathbf{1}_{|\mathcal{B}|}$ is weighting for $[\mathcal{A}]$. Finally, we have
$$\mathbf{1}_{|\mathcal{B}|}^*[\mathcal{B}]^+\mathbf{1}_{|\mathcal{B}|}
=\mathbf{1}_{|\mathcal{A}|}^*[R][\mathcal{B}]^+\mathbf{1}_{|\mathcal{B}|}
=\mathbf{1}_{|\mathcal{A}|}^*[\mathcal{A}]^+[\mathcal{A}][R][\mathcal{B}]^+\mathbf{1}_{|\mathcal{B}|}
$$
$$=\mathbf{1}_{|\mathcal{A}|}^*[\mathcal{A}]^+[L]^*[\mathcal{B}][\mathcal{B}]^+\mathbf{1}_{|\mathcal{B}|}
=\mathbf{1}_{|\mathcal{A}|}^*[\mathcal{A}]^+[L]^*\mathbf{1}_{|\mathcal{B}|}
=\mathbf{1}_{|\mathcal{A}|}^*[\mathcal{A}]^+\mathbf{1}_{|\mathcal{A}|}.$$
Hence $\chi (\mathcal{A})= \chi (\mathcal{B})	$.
\end{proof}

Let $\mathcal{A}$ be a category and  $F:\mathcal{A}\rightarrow \mathbf{Cat}$ be a pseudofunctor. Then the \textbf{Grothendieck construction} $G(F)$ is the category that has as objects all pairs $(a,x)$ where $a$ is an object in  $\mathcal{A}$ and $x$ is an object in $F(a)$, and morphisms  from $(a,x)$ to $(b,y)$ are pairs $(f,\zeta)$ where $f:a\rightarrow b$ is a morphism in $\mathcal{A}$ and $\zeta:F(f)(x)\rightarrow y$ is a morphism in $F(b)$.

For the rest of this section, let $\mathcal{A}$ be a finite category with object set 
$$\mathrm{Ob}(\mathcal{A})=\{a_1,a_2,\dots a_m\}.$$ 
Let $F:\mathcal{A}\rightarrow \mathbf{Cat}$ be a pseudofunctor. Assume that  for  every $i$ in $\{1,2,\dots m\}$, the category $F(a_i)$ has object set $$\mathrm{Ob}(F(a_i))=\{x_{i1},x_{i2},\dots x_{in_i}\}$$ for some natural number $n_i$. Then $G(F)$ has $n_1+n_2+\dots +n_m$ many objects and objects of $G(F)$ are ordered as follows:
$(a_1,x_{11})$, $(a_1,x_{12})$, $\dots $, $(a_1,x_{1n_1})$, $(a_2,x_{21})$, $\dots $, $(a_2,x_{2n_2})$, $\dots $, $(a_m,x_{m1})$, $\dots $, $(a_m,x_{mn_m})$.

Let $U:\mathbf{Cat}\rightarrow \mathbf{Set}$ denote the functor that sends a category to its set of objects.  Let $T:\mathbf{Set}\rightarrow \mathbf{Cat}$ denote the discrete category functor.  Define $L_1(F):\mathcal{A}\rightarrow \mathbf{Cat} $ as follows: $$L_1(F)=T\circ U\circ F.$$ Notice that there exists a natural transformation $i:T\circ U\Rightarrow \mathrm{Id}_{\mathbf{Cat}}$ given by inclusion on each component. Define $L_2(F):T\circ U(\mathcal{A})\rightarrow \mathbf{Cat} $ as follows:  $$L_2(F)=F\circ i_{\mathcal{A}}.$$
Then we have the following lemmas:

\begin{lemma} \label{lemma_for_G_L_2_F}
	$[G(L_2(F))]=\mathrm{diag}\left([F(a_1)],[F(a_2)],\dots,[F(a_m)]\right)$.
\end{lemma}	
\begin{proof} Given two objects  $a$, $a'$ in $\mathcal{A}$ and an object $x$ in $F(a)$ and  $x'$ in $F(a')$, we have
	$$G(L_2(F))((a,x),(a',x'))\cong F(a)(x,x')$$ when
	$a=a'$ and $$G(L_2(F))((a,x),(a',x'))=\emptyset$$ when  $a\neq a'$.
\end{proof}

\begin{lemma}\label{lemma_of_seperation}
	$[G(F)]=\left[G(L_1(F))\right]\left[G(L_2(F))\right].$
\end{lemma}	
\begin{proof}  Given two objects  $a$, $a'$ in $\mathcal{A}$ and an object $x$ in $F(a)$ and $x'$ in $F(a')$, we have
	 $$G(F)((a,x),(a',x'))\cong \coprod _{y\text{ object in } F(a')} \{\, f\in \mathcal{A}(a,a')\, |\, f(x)=y\, \}\times F(a')(y,x').$$
Now notice that
$$|\{\, f\in \mathcal{A}(a,a')\, |\, f(x)=y\, \}|=|G(L_1(F))((a,x),(a',y))|.$$
Now by Lemma \ref{lemma_for_G_L_2_F}, we have
$$|G(L_2(F))((a',y),(a',x'))|=|F(a')(y,x')|.$$
Hence we have
$$|G(F)((a,x),(a',x'))|=\sum_{y\text{ object in } F(a')}|G(L_1(F))((a,x),(a',y))|G(L_2(F))((a',y),(a',x'))||.$$
Also by Lemma \ref{lemma_for_G_L_2_F}, we have
$$|G(L_2(F))((\widetilde{a},\widetilde{y}),(a',x'))|=0$$
for any object $\widetilde{a}$ in $\mathcal{A}$ and an object $\widetilde{y}$ in $F(\widetilde{a})$ when $\widetilde{a}\neq a'$. Hence we get
$$[G(F)]=\left[G(L_1(F))\right]\left[G(L_2(F))\right].$$
\end{proof}
Let $v_1$, $v_2$, $\dots$, $v_m$ be column vectors of possibly different sizes. Then we write $C(v_1,v_2,\dots , v_m)$ for the following column vector
$$C(v_1,v_2,\dots , v_m)=\left[\begin{array}{c}
	v_1 \\
	v_2 \\
	\vdots  \\
	v_m 		
\end{array}\right]$$
whose number of rows is equal to the sum of the number rows of $v_1$, $v_2$, $\dots$, $v_m$. Note that if $v_1$, $v_2$, $\dots$, $v_m$ are all rational numbers then $C(v_1,v_2,\dots , v_m)$ is an $m\times 1$-matrix.

\begin{lemma} \label{lemma_for_G_L_1_F} Let $\mu_1,\mu_2,\dots, \mu_m,\lambda_1,\lambda_2,\dots, \lambda_m $ be rational numbers. Assume $$[\mathcal{A}]C\left(\mu_1,\mu_2,\dots, \mu_m\right)=
C\left(\lambda_1,\lambda_2,\dots, \lambda_m\right).$$ Then we have
$$\left[G(L_1(F))\right]C\left(\mu_1\mathbf{1}_{n_1},\mu_2\mathbf{1}_{n_2},\dots, \mu_m\mathbf{1}_{n_m}\right)=
C\left(\lambda_1\mathbf{1}_{n_1},\lambda_2\mathbf{1}_{n_2},\dots, \lambda_m\mathbf{1}_{n_m}\right).$$	
\end{lemma}	
\begin{proof} Given two objects  $a$, $a'$ in $\mathcal{A}$ and an object $x$ in $F(a)$ and  $x'$ in $F(a')$  we have
	$$\sum _{x'\text{ object in }F(a')}|G(L_1(F))((a,x),(a',x'))|=|\mathcal{A}(a,a')|.$$
Hence the result follows.	
\end{proof}

\begin{remark}\label{remark_about_ext_order}
When $\mathcal{A}$ is a poset, we can choose a total order on objects of $\mathcal{A}$ that extends the partial order on it. Hence, whenever $\mathcal{A}$ is a poset, without loss of generality we can assume that there exists a morphism from $a_i$ to $a_j$ in $\mathcal{A}$ only if $i\leq j$.
\end{remark}

\begin{lemma}\label{lemma_about_existence_of_inv_M}
	If $\mathcal{A}$ is a poset then there exists an invertible matrix $M$ such that $$M\,[G(L_2(F))]=[G(L_1(F))][G(L_2(F))].$$
\end{lemma}
\begin{proof} Assume that $\mathcal{A}$ is a poset. By Remark \ref{remark_about_ext_order}, we can assume that the chosen total order $a_1$, $a_2$, $\dots $, $a_m$ on the set of objects of $\mathcal{A}$ extends the partial order on it given by the poset structure on $\mathcal{A}$. For $i$ in $\{1,2\}$, let $M_i$ denote the matrix $[G(L_i(F))]$. Hence to prove the lemma we need to find an invertible matrix $M$ such that $(M-M_1)M_2=0$. Consider these matrices as functions from $\mathrm{Ob}(G(F))\times \mathrm{Ob}(G(F))$ to $\mathbb{Z}$. Let $M$ be the matrix defined as follows:
	\[
	M((a,x),(a',x')) =
	\begin{cases}
		M_1((a,x),(a',x')) & \text{if } a \neq a' \\
		1 & \text{if } (a,x)=(a',x') \\
		0 & \text{otherwise }
	\end{cases}
	\]
	for objects $(a,x)$, $(a',x')$ in $G(F)$. First notice that $M$ is an upper triangular matrix and all of its diagonal entries are equal to $1$. Hence $M$ is an invertible matrix. Second notice that for each object $a$ in $\mathcal{A}$, there exists a natural isomorphism from $\mathrm{id}_{F(a)}$ to $F(\mathrm{id}_{a})$, since $F$ is a pseudofunctor. In particular, this means for each object $(a,x)$ in $G(F)$ there exists an isomorphism $\sigma_{a,x}:x\rightarrow F(\mathrm{id}_{a})(x)$ in the category $F(a)$. Therefore we have
	$$M_2((a,x),(a',x'))=M_2((a,F(\mathrm{id}_{a})(x)),(a',x'))$$
	and
	$$
	(M-M_1)((a,x),(a',x')) =
	\begin{cases}
		1 & \text{if } a=a' \text{ and } x = x' \text{ and } x'\neq F(\mathrm{id}_{a})(x) \\
		-1 & \text{if } a=a' \text{ and } x\neq x' \text{ and } x' =F(\mathrm{id}_{a})(x) \\
		0 & \text{otherwise}
	\end{cases}
	$$
	for every object $(a,x)$, $(a',x')$ in $G(F)$. Thus we have
	$$
	\begin{aligned}
		(M-M_1 )M_2((a,x),(a',x')) & = \sum_{y\in \{x,F(\mathrm{id}_{a})(x))\}} (M-M_1)((a,x),(a,y))M_2((a,y),(a',x')) \\
		& = M_2((a,x),(a',x')) - M_2((a,F(\mathrm{id}_{a})(x)),(a',x')) \\
		& = 0
	\end{aligned}
	$$
	Hence we have proved that $(M-M_1)M_2=0$.
\end{proof}

The next lemma is an immediate corollary of the previous one.

\begin{lemma} \label{lemma_about_row_spaces}
	If $\mathcal{A}$ is a poset then $\mathrm{Row}([G(L_1(F))][G(L_2(F))])=\mathrm{Row}([G(L_2(F))])$.
\end{lemma}
\begin{proof}
	By Lemma \ref{lemma_about_existence_of_inv_M}, we know that there exists an in invertible matrix $M$ such that $M\,[G(L_2(F))]=[G(L_1(F))][G(L_2(F))].$
	We can write $M$ as a product of elementary matrices since $M$ is invertible. Moreover multiplication by an elementary matrix from left is same as performing a row operation. Therefore the matrix $[G(L_2(F))]$ is row equivalent to $M\,[G(L_2(F))]$. Hence we have $$\mathrm{Row}([G(L_1(F))][G(L_2(F))])=\mathrm{Row}([G(L_2(F))]).$$
\end{proof}

\begin{lemma}\label{lemma:inc_exc_weight} Let $C(\lambda _1,\lambda_2,\dots,\lambda_m)$ be a weighting for $[\mathcal{A}]$ and  $v_i $ a weighting for $[F(a_i)]$ for every $i$ in $\{1,2,\dots m\}$.  Then $C(\lambda_1 v_1,\lambda_2 v_2,\dots, \lambda_m v_m)$ is a weighting for $G(F)$.
\end{lemma}
\begin{proof} Let  $C(\lambda_1,\lambda_2,\dots,\lambda_m)$ be a weighting for $[\mathcal{A}]$  and  $v_i $ a weighting for $[F(a_i)]$ for every $i$ in $\{1,2,\dots m\}$.
By Lemma \ref{lemma_for_G_L_2_F} we have $$[G(L_2(F))]=\mathrm{diag}\left([F(a_1)],[F(a_2)],\dots,[F(a_m)]\right).$$
Hence $$[G(L_2(F))]C\left(\lambda_1 v_1,\lambda_2 v_2,\dots, \lambda_m v_m)\right)=
C\left(\lambda_1 \mathbf{1}_{n_1},\lambda_2\mathbf{1}_{n_2},\dots, \lambda_m\mathbf{1}_{n_m}\right).$$	
By Lemma \ref{lemma_for_G_L_1_F}, we have
$$\left[G(L_1(F))\right]C\left(\lambda_1\mathbf{1}_{n_1},\lambda_2\mathbf{1}_{n_2},\dots, \lambda_m\mathbf{1}_{n_m}\right)=C\left(\mathbf{1}_{n_1},\mathbf{1}_{n_2},\dots, \mathbf{1}_{n_m}\right)=\mathbf{1}_{\mathbf{|G(F)|}}.$$
By Lemma \ref{lemma_of_seperation}, we have $$[G(F)]=[G(L_1(F))][G(L_2(F))].$$
Hence we have
$$[G(F)]C\left(\lambda_1 v_1,\lambda_2 v_2,\dots, \lambda_m v_m\right)=\mathbf{1}_{\mathbf{|G(F)|}}.$$
Hence 	$C(\lambda_1 v_1,\lambda_2 v_2,\dots, \lambda_m v_m)$ is a weighting for $G(F)$.
\end{proof}

Let $\mathcal{A}$ be a finite category with object set  $\{a_1,a_2,\dots a_m\}$.
For a pseudofunctor $F:\mathcal{A}\rightarrow \mathbf{Cat}$, let $\chi (F)$ denote the $1\times m$-matrix
$$\chi (F)=\left[\,\, \chi(F(a_1))\,\, \chi(F(a_2))\, \cdots \, \chi(F(a_m))\,\, \right].$$

\begin{theorem}\label{thm:inc_exc_with_coweighting}  Let $\mathcal{A}$ be a finite category and  $F:\mathcal{A}\rightarrow \mathbf{Cat}$ a pseudofunctor such that for every object $a$ in $\mathcal{A}$, the category $F(a)$ is finite. Assume that  $[G(F)]$ has a coweighting, $[\mathcal{A}]$ has weighting, and for every object $a$ in $\mathcal{A}$, $[F(a)]$ has a weighting and a coweighting. Then we have
	$$\chi (G(F))=\chi (F)[\mathcal{A}]^+\mathbf{1}_{|\mathcal{A}|}	.$$
\end{theorem}
\begin{proof} Since $[\mathcal{A}]$ has a weighting, by Lemma \ref{lemma_for_existence_of_weighting}  we know that $[\mathcal{A}]^+\mathbf{1}_{|\mathcal{A}|}$ is  a weighting  for $[\mathcal{A}]$.
For some rational numbers $\lambda_1,\ \lambda_2,\dots ,\ \lambda_m$, we have  $$[\mathcal{A}]^+\mathbf{1}_{|\mathcal{A}|}=C(\lambda_1,\lambda_2,\dots, \lambda_m).$$	
Assume that $v_i $ is a weighting for $[F(a_i)]$ for every $i$ in $\{1,2,\dots m\}$.
Then by Lemma \ref{lemma:inc_exc_weight}, $C(\lambda_1 v_1,\lambda_2 v_2,\dots, \lambda_m v_m)$ is a weighting for $G(F)$. Hence by Lemma \ref{lemma_for_SLS}, we have
$$\chi (G(F))=\mathbf{1}_{|G(F)|}^*C(\lambda_1 v_1,\lambda_2 v_2,\dots, \lambda_m v_m).$$
Since for every object $a$ in $\mathcal{A}$, $[F(a)]$ has a weighting and a coweighting, we have
$$\chi(F(a_i))=\mathbf{1}_{n_i}^*v_i$$ by Lemma \ref{lemma_for_SLS}. Since $\mathbf{1}_{|G(F)|}^*=[\,\,\mathbf{1}_{n_1}^*\,\,\mathbf{1}_{n_2}^*\,\,\dots \,\, \mathbf{1}_{n_m}^*\,\,]$, we have
$$\chi (G(F))=\sum _{i=1}^{m}\lambda_i \mathbf{1}_{|F(a_i)|}^*v_i=\sum _{i=1}^{m}\lambda_i \chi(F(a_i))= \chi(F)[\mathcal{A}]^+\mathbf{1}_{|\mathcal{A}|}	.$$
\end{proof}

\begin{theorem}\label{thm:inc_exc}   Let $\mathcal{A}$ be a finite poset and  $F:\mathcal{A}\rightarrow \mathbf{Cat}$ a pseudofunctor such that for every object $a$ in $\mathcal{A}$, the category $F(a)$ is finite. Assume that  for every object $a$ in $\mathcal{A}$, $[F(a)]$ has a weighting. Then we have
	$$\chi (G(F))=\chi (F)[\mathcal{A}]^+\mathbf{1}_{|\mathcal{A}|}.$$
\end{theorem}
\begin{proof}
	Since $\mathcal{A}$ is a finite poset, the matrix $[\mathcal{A}]$ is invertible and hence it has weighting. Therefore  $[\mathcal{A}]^+\mathbf{1}_{|\mathcal{A}|}$ is  a weighting for $[\mathcal{A}]$  by Lemma \ref{lemma_for_existence_of_weighting}. For some rational numbers $\lambda_1,\ \lambda_2,\dots ,\ \lambda_m$, we have  $$[\mathcal{A}]^+\mathbf{1}_{|\mathcal{A}|}=C(\lambda_1,\lambda_2,\dots, \lambda_m).$$
	For every object $a$ in $\mathcal{A}$, we know that $[F(a)]^+\mathbf{1}_{|F(a)|}$ is a weighting for $[F(a)]$  by  Lemma \ref{lemma_for_existence_of_weighting} since $[F(a)]$ has a weighting.
	Define $v_i=[F(a_i)]^+\mathbf{1}_{|F(a_i)|}$ for $i$ in $\{1,2,\dots,m\}$. Then by Lemma \ref{lemma:inc_exc_weight}, $C(\lambda_1 v_1,\lambda_2 v_2,\dots, \lambda_m v_m)$ is a weighting for $G(F)$. Notice that $C(0,\dots,0,v_i,0,\dots 0)^*$ is vector in $\mathrm{Row}(G(L_2(F)))$ due to Lemma \ref{lemma_for_G_L_2_F} and the fact that $v^*_i$ is in $\mathrm{Row}([F(a_i)])$ for each $i \in \{1,2, \dots , m\}$. Now by Lemma \ref{lemma_about_row_spaces}
	we know that $$\mathrm{Row}([G(L_1(F))][G(L_2(F))])=\mathrm{Row}([G(L_2(F))])$$ and by Lemma \ref{lemma_of_seperation} we know that
	$$[G(F)]=[G(L_1(F))][G(L_2(F))].$$
	Hence $C(\lambda_1 v_1,\lambda_2 v_2,\dots, \lambda_m v_m)^*$ is in $\mathrm{Row}(G(F))$.
	Therefore by Lemma \ref{lemma_about_moore_penrose_inverse} we have $$[G(F)]^+\mathbf{1}_{G(F)}=C(\lambda_1 v_1,\lambda_2 v_2,\dots, \lambda_m v_m).$$
	So as in the pevious proof we again have
    $$\chi (G(F))=\mathbf{1}_{|G(F)|}^*C(\lambda_1 v_1,\lambda_2 v_2,\dots, \lambda_m v_m).$$	
    Thus we are done by the same argument.	
\end{proof}

\begin{example}
Notice that the category $$\mathcal{P} =\{ b\leftarrow a \rightarrow c \} $$
is a poset and
$$[\mathcal{P}]=
\left[\begin{array}{rrr}
	1  & 1 & 1\\
	0  &  1 & 0 \\
	0 & 0  & 1
\end{array}\right]$$ and
$$[\mathcal{P}]^+=
\left[\begin{array}{rrr}
	1  & -1 & -1\\
	0  &  1 & 0 \\
	0 & 0  & 1
\end{array}\right].$$
Now consider any pseudofunctor $F:\mathcal{P}\rightarrow \mathbf{Cat}$  such that for every object $x$ in $\mathcal{P}$, the category $F(x)$ is finite and $[F(x)]$ has a weighting. Then we have
$$\chi (G(F))=\chi (F)[\mathcal{P}]^+\mathbf{1}_{|\mathcal{P}|}$$
 by Theorem \ref{thm:inc_exc}. Moreover, we have $$\chi (F)=[\,\,\chi (F(a))\,\,\chi (F(b))\,\,\chi (F(c))\,\, ]$$
and
$$[\mathcal{P}]^+\mathbf{1}_{|\mathcal{P}|}=\left[\begin{array}{rrr}
	1  & -1 & -1\\
	0  &  1 & 0 \\
	0 & 0  & 1
\end{array}\right]\left[\begin{array}{r}
 1\\
 1 \\
 1
\end{array}\right]=\left[\begin{array}{r}
-1\\
1 \\
1
\end{array}\right].$$
Hence
$$\chi (F)=[\,\,\chi (F(a))\,\,\chi (F(b))\,\,\chi (F(c))\,\, ]\left[\begin{array}{r}
	-1\\
	1 \\
	1
\end{array}\right]=\chi (F(b))+\chi (F(c))-\chi (F(a)).$$

\end{example}

Theorem  \ref{thm:inc_exc} doesn't imply that the Euler measure preserves products, so we prove this separately.

\begin{theorem}\label{thm:product} Let $\mathcal{A}$, $\mathcal{B}$ be two finite categories. Then we have
	$$\chi (\mathcal{A}\times \mathcal{B})=\chi (\mathcal{A})\chi(\mathcal{B})	.$$
\end{theorem}
\begin{proof} Let $\mathcal{A}$, $\mathcal{B}$ be two finite categories.
	Then $$[\mathcal{A}\times \mathcal{B}]=[\mathcal{A}]\otimes[\mathcal{B}]$$
	where $\otimes$ denote Kronecker product.	
	Hence
	$$[\mathcal{A}\times  \mathcal{B}]^+=[\mathcal{A}]^+\otimes [\mathcal{B}]^+.$$
	Also notice
	$$\mathbf{1}_{|\mathcal{A}\times \mathcal{B}|}=\mathbf{1}_{|\mathcal{A}|}\otimes \mathbf{1}_{|\mathcal{B}|}.$$
Since in general we know $(M \otimes N)(L\otimes K) = (ML) \otimes (NK)$, we have   $$\mathbf{1}_{|\mathcal{A}\times \mathcal{B}|}^*[\mathcal{A}\times  \mathcal{B}]^+\mathbf{1}_{|\mathcal{A}\times  \mathcal{B}|}=
	\mathbf{1}_{|\mathcal{A}}^*[\mathcal{A}]^+\mathbf{1}_{|\mathcal{A}|}\otimes\mathbf{1}_{|\mathcal{B}|}^*[ \mathcal{B}]^+\mathbf{1}_{ \mathcal{B}|}.$$
	Here the operation $\otimes $ on the right-hand side of the equality above is in between two $1\times 1$-matrices. Hence it just corresponds to the multiplication of rational numbers.
\end{proof}

Now we  show that Euler measure is additive under disjoint unions.

\begin{theorem}\label{thm:disjoint} Let $\mathcal{A}$, $\mathcal{B}$ be two finite categories. Then we have
	$$\chi (\mathcal{A}\sqcup  \mathcal{B})=\chi (\mathcal{A})+\chi(\mathcal{B}).	$$
\end{theorem}
\begin{proof} Let $\mathcal{A}$, $\mathcal{B}$ be two finite categories.
	Then $$[\mathcal{A}\sqcup  \mathcal{B}]=\mathrm{diag}([\mathcal{A}],[\mathcal{B}]).$$
	Hence
	$$[\mathcal{A}\sqcup  \mathcal{B}]^+=\mathrm{diag}([\mathcal{A}]^+,[\mathcal{B}]^+).$$
Also notice
 $$\mathbf{1}_{|\mathcal{A}\sqcup  \mathcal{B}|}=C(\mathbf{1}_{|\mathcal{A}|},\mathbf{1}_{|\mathcal{B}|}).$$
So $$\mathbf{1}_{|\mathcal{A}\sqcup  \mathcal{B}|}^*[\mathcal{A}\sqcup  \mathcal{B}]^+\mathbf{1}_{|\mathcal{A}\sqcup  \mathcal{B}|}=
\mathbf{1}_{|\mathcal{A}}^*[\mathcal{A}]^+\mathbf{1}_{|\mathcal{A}|}+\mathbf{1}_{|\mathcal{B}|}^*[ \mathcal{B}]^+\mathbf{1}_{ \mathcal{B}|}.$$
\end{proof}

There are many artificially created examples of categories with adjacency matrix having a weighting and no coweighting. For the following examples, we will use Corollary 4.2 in \cite{berger_leinster}.

\begin{example}\label{EX1}
	Let $\mathcal{C}_1$ be a category with adjacency matrix $$[\mathcal{C}_1]=\left[\begin{array}{rr}
		3 & 2 \\
		3 & 2
	\end{array}\right].$$Clearly $[\mathcal{C}_1]$ admits a weighting and no coweighting. By using the algorithim in the Appendix, we find its Moore-Penrose inverse as $$[\mathcal{C}_1]^+ =\left[\begin{array}{rrr}
	3/26 & 3/26 \\
	1/13 & 1/13
\end{array}\right].$$ So its Euler measure is $\chi(\mathcal{C}_1) = 5/13.$
\end{example}

\begin{example}\label{EX2}
Consider the category $\mathcal{C}_1$ in Example \ref{EX1}. Let $\mathcal{C}_2$ be the category obtained by adding an isomorphic copy of the second object of $\mathcal{C}_1$ to itself as a third object. Then $$[\mathcal{C}_2] =\left[\begin{array}{rrr}
	3 & 2 & 2 \\
	3 & 2 & 2 \\
	3 & 2 & 2
\end{array}\right].$$ Observe that $[\mathcal{C}_2]$ has a weighting and no coweighting. Then by using the algorithm to find Moore-Penrose inverse, we get $$[\mathcal{C}_2]^+ =\left[\begin{array}{rrr}
	1/17 & 1/17 & 1/17 \\
	2/51 & 2/51 & 2/51 \\
	2/51 & 2/51 & 2/51
\end{array}\right].$$ So its Euler measure $\chi(\mathcal{C}_2) = 7/17.$
\end{example}

Clearly $\mathcal{C}_1$ and $\mathcal{C}_2$ are equivalent categories. However, $\chi(\mathcal{C}_1) \neq \chi(\mathcal{C}_2)$. Thus the Euler measure $\chi$ is not invariant under equivalence of categories.

\begin{example}\label{EX3}
	Consider the catories $\mathcal{C}_1$ and $\mathcal{C}_2$ in Example \ref{EX2}. Let $F: \mathcal{C}_1\rightarrow \mathbf{Cat}$ be a functor mapping the first object to the terminal category $\mathbf{*}$ and second object to the category $\mathcal{B}$ whose adjacency matrix is $$[\mathcal{B}]=\left[\begin{array}{rr}
1 & 1 \\
1 & 1
\end{array}\right].$$ Observe that $$[G(F)] =\left[\begin{array}{rrr}
	3 & 2 & 2 \\
	3 & 2 & 2 \\
	3 & 2 & 2
\end{array}\right]$$ and it is independent of the choice of the image of the morphisms of the category $\mathcal{C}_1$ under $F$. Thus $\chi(G(F))= 7/17$ as in Example \ref{EX2}. On the other hand $$\chi (F)=\left[\,\, \chi(*)\,\, \chi(\mathcal{B})\,\, \right]=\left[\,\,  1 \,\, 1 \,\, \right]$$ and $$[\mathcal{C}_1]^+ =\left[\begin{array}{rrr}
3/26 & 3/26 \\
1/13 & 1/13
\end{array}\right].$$
By applying the formula in Theorem \ref{thm:inc_exc}, we get $$\chi (F)[\mathcal{C}_1]^+\mathbf{1}_{|\mathcal{C}_1|} =\left[\,\,  1 \,\, 1 \,\, \right] \left[\begin{array}{rrr}
	3/26 & 3/26 \\
	1/13 & 1/13
\end{array}\right] \left[\begin{array}{r}
1\\
1
\end{array}\right] = \frac{5}{13} \neq \frac{7}{17} = \chi(G(F)).$$ Hence Theorem \ref{thm:inc_exc} is not true without the assumption that the index category $\mathcal{A}$ is a poset.

	\end{example}

\section{Appendix}\label{sec:appendix}
The goal of this appendix is to explain why $\chi(\mathcal{A})$ is a rational number for every category $\mathcal{A}$. For this purpose we give an algorithm for computing the Moore-Penrose inverse (also known as pseudo inverse) of a square matrix. 

There are many different algorithms to compute Moore-Penrose inverse in the literature. Some of the most popular methods for computing the Moore-Penrose inverse are based on the singular value decomposition (SVD) of matrices as the method used by Penrose in \cite{penrose}. However, by using this method we cannot prove that the entries of the Moore-Penrose inverse of the incidence matrix of a category are rational.

Another common tool for this purpose is the full rank decomposition of matrices (See Theorem 5, page: 48 in \cite{israel}). Using this method one can see that the entries of the Moore-Penrose inverse of the incidence matrix of a category are all rational. 

There are also several algorithms (See \cite{chen}, \cite{murray-lasso}, \cite{davis} and \cite{ji}) which compute the Moore-Penrose inverse by generalizing the normal equation method. We also give an algorithm that generalizes normal equation method here. Although this is not a new method, we give it here to stress out why the entries of the matrices are all rational numbers that appear during the computations.

Let $C_m$ denote the vector space of all $m\times 1$-matrices with rational entries. Let $M$ be an $m\times m$ matrix with rational entries. Then the vector spaces 
$\mathrm{Null}(M)=\{\,x\in C_m\,|\,Mx=0\,\}$, 
$\mathrm{im}(M)=\{\,Mx\,|\,x\in C_m\,\}$, $\mathrm{Row}(M)=\{\,x^*\,|\,x\in \mathrm{im}(M^*)\,\}$,
and $\mathrm{im}(M)^{\perp} = \mathrm{Null}(M^*)$ all have bases that consists of matrices with rational entries.
Here we will discuss an algorithm to compute the Moore-Penrose inverse $M^+$ of a matrix $M$. The main theorem of this subsection is the following.

\begin{theorem}
	Let $M$ be an $m\times m$-matrix. Assume that $\{r_1,r_2,\dots,r_k\}$ is a basis for $\mathrm{Row}(M)$ and $\{s_1,s_2,\dots,s_{m-k}\}$ is a basis for $\mathrm{im}(M)^{\perp}$. Then the matrix $$\left[\,\, Mr_1^* \,\, Mr_2^* \, \cdots  \,  Mr_k^* \,\, s_1 \,\, s_2 \, \cdots \, s_{m-k} \,\,\right]$$ is invertible and we have
	$$M^+=\left[\,\, r_1^* \, r_2^* \, \cdots \,  r_k^* \,\, \mathbf{0}\,\, \right]\left[\,\, Mr_1^* \,\, Mr_2^* \, \cdots \, Mr_k^* \,\, s_1 \,\, s_2 \, \cdots \, s_{m-k} \, \right]^{-1}$$
	where $\mathbf{0}$ is zero matrix of size $m\times (m-k)$.
\end{theorem}	
First of all the matrix $\left[\,\, Mr_1^* \,\, Mr_2^* \, \cdots Mr_k^* \,\, s_1 \,\, s_2 \, \cdots \, s_{m-k} \,\,\right]$ is invertible because $\{Mr_1^* , Mr_2^* , \dots , Mr_k^*\}$ is a basis for $\mathrm{im}(M)$ and $\{s_1 ,s_2 , \dots ,s_{m-k}\}$ is a basis for $\mathrm{im}(M)^{\perp }$. The rest of the proof of the above theorem directly follows from the next four lemmas. First we fix some notation that is used in the rest of this subsection.
Assume that $M$ is an $m\times m$-matrix,  $\{s_1,s_2,\dots,s_{m-k}\}$ is a basis for $\mathrm{im}(M)^{\perp}$, and $\{r_1,r_2,\dots,r_k\}$ is a basis for $\mathrm{Row}(M)$. Then let   $A_M$ be the $m\times m$-matrix $$A_M=\left[\,\, r_1^* \, r_2^* \,\cdots \, r_k^* \,\, \mathbf{0}\,\, \right]$$ where $\mathbf{0}$ is zero matrix of size $m\times (m-k)$. Also let $B_M$ be the  $m\times m$-matrix $$B_M=\left[\,\, Mr_1^* \,\, Mr_2^* \, \cdots \, Mr_k^* \,\, s_1 \,\, s_2 \, \cdots \, s_{m-k} \,\,\right]^{-1}.$$
Now notice that the above theorem says $M^+=A_MB_M$. Hence the next four lemmas finishes the proof of the above theorem.

Note that two $n\times m$ matrices are equal to each other if we consider them as linear functions whose domains are $C_m$ and they agree on a basis of $C_m$.
Let $\{e_1,e_2,\dots e_m\}$ be the standard basis of the vector space $C_m$
and $\{t_1,t_2,\dots,t_{m-k}\}$ a basis for $\mathrm{Null}(M)$. Then
$$\beta_1= \{ r_1^* , r_2^* , \dots, r_k^* , t_1, t_2 , \dots , t_{m-k}\}$$
is basis for $C_m$.
\begin{lemma}$MA_MB_MM=M.$
\end{lemma}	
\begin{proof} We know $\beta _1$ is a basis for $C_m$ and we have
	$$MA_MB_MMr_i^*=MA_Me_i=Mr_i^*$$ 	
	for $1\leq i\leq k$ and
	$$MA_MB_MMt_i=MA_Me_{k+i}=M0=0=Mt_i $$ 	
	for $1\leq i\leq m-k$.
\end{proof}

Now notice that the set
$$\beta_2=\{ Mr_1^* , Mr_2^* , \dots, Mr_k^* , s_1, s_2 , \dots , s_{m-k}\}$$
is another basis for $C_m$.

\begin{lemma}$A_MB_MMA_MB_M=A_MB_M$.
\end{lemma}	
\begin{proof} We know $\beta _2$ is a basis for $C_m$ and we have
	$$A_MB_MMA_MB_MMr_i^*=A_MB_MMA_Me_i=A_MB_MMr_i^*=A_MB_MMr_i^*$$ 	
	for $1\leq i\leq k$ and
	$$A_MB_MMA_MB_Ms_i=A_MB_MMA_Me_{k+i}=A_MB_MM0=0=A_Me_{k+i}= A_MB_Ms_i$$ 	
	for $1\leq i\leq m-k$.
\end{proof}

The above two lemmas prove that $A_MB_M$ satisfies the first two equations. We know that if $Z$ is an invertible matrix and $X$ is a matrix such that $Z^* XZ = Z^* X^* Z$, then $X=X^*$. Now we prove the third equation.

\begin{lemma}$(A_MB_MM)^*=A_MB_MM$.
\end{lemma}	
\begin{proof} We know $\beta _1$ is a basis of $C_m$ and we have
	$$r_jA_MB_MMr_i^*=r_jA_Me_i=r_jr_i^*$$
	and
	$$r_j(A_MB_MM)^*r_i^*=(A_MB_MMr_j^*)^*r_i^*=(A_Me_j)^*r_i^*=(r_j^*)^*r_i^*=r_jr_i^*$$ 	
	for $1\leq i\leq k$ and
	$$Mt_i=0\text{\ \  and \ \ }t_i^*A_M=0$$
	for $1\leq i\leq m-k$.
\end{proof}

Finally we prove the last equation

\begin{lemma}$(MA_MB_M)^*=MA_MB_M$
\end{lemma}	
\begin{proof} We know $\beta _2$ is a basis of $C_m$ and we have
	$$(Mr_j^*)^*MA_MB_M(Mr_i^*)=(Mr_j^*)^*Mr_i^*$$
	and
	$$ (Mr_j^*)^*(MA_MB_M)^*(Mr_i^*)=(MA_MB_MMr_j^*)^*r_iM^*=(Mr_j^*)^* Mr_i^*$$ 	
	for $1\leq i\leq k$ and
	$$A_MB_Ms_i=A_Me_{k+i}=0 \text{\ \  and \ \ }s_i^*M=0 $$ 	
	for $1\leq i\leq m-k$.
\end{proof}

\section*{Acknowledgements}
The authors wish to express their gratitude to the referee for providing valuable comments and suggestions that have significantly enhanced the quality of this paper. The first author would like to thank Matthew Gelvin for helpful discussions and motivating to study this topic.


\end{document}